\documentclass[
final
]{dmtcs-episciences}

\usepackage[utf8]{inputenc}

\usepackage[square,numbers]{natbib}

\usepackage{amssymb,amsmath,amsthm}

\usepackage{mathdots}
\usepackage{cases}
\usepackage{csquotes}
\usepackage{graphicx,psfrag}
\usepackage{microtype}
\usepackage{xspace}
\usepackage{thmtools}
\usepackage{marvosym}
\usepackage{tabularx}
\usepackage{color,soul}
\usepackage{hyperref,cleveref}
\usepackage[normalem]{ulem}
\usepackage{paralist,enumerate}					  

\usepackage[vlined,linesnumbered,ruled]{algorithm2e}
\usepackage{algorithmic}
\usepackage{float}
\usepackage{subfig}

\usepackage{tikz}
\usetikzlibrary{positioning}
\usetikzlibrary{shapes}
\usetikzlibrary{shapes.symbols,patterns}
\usetikzlibrary{calc,through,backgrounds}
\usetikzlibrary{decorations.pathreplacing}
\usetikzlibrary{arrows,decorations.pathmorphing,backgrounds,positioning,fit,petri}

\usepackage{orcidlink}

\newtheorem{theorem}{Theorem}
\newtheorem{lemma}[theorem]{Lemma}

\newtheorem{observation}[theorem]{Observation}
\newtheorem{corollary}[theorem]{Corollary}

\newcommand{\paren}[1]{\left( #1 \right)}

\newcommand{\set}[1]{\left\{ #1 \right\}}

\newcommand{\ceil}[1]{\left\lceil #1 \right\rceil}

\newcommand{\abs}[1]{\left| #1 \right|}

\newcommand{\hd}{{\hat d}}
\newcommand{\td}{{\tilde d}}

\newcommand{\bidx}[1]{\text{box}(#1)}

\newcommand{\vol}[1]{{\sum{#1}}}
\newcommand{\conj}[1]{\widetilde{#1}}

\newcommand{\ignore}[1]{}
\newcommand{\remove}[1]{}

\def\bp{{\tt BP}}

\def\seqd{d}

\def\cP{\mathcal{P}}
\def\bp{{\tt BP}}

\def\maxmult{\mbox{\sf MaxMult}}
\def\maxmultbi{\mbox{\sf MaxMult}^{bi}}

\def\totmult{\mbox{\sf TotMult}}

\def\totmultbi{\mbox{\sf TotMult}^{bi}}

\newcommand{\totg}{tot-graphic\xspace}
\newcommand{\maxg}{max-graphic\xspace}
\newcommand{\totbi}{tot-bigraphic\xspace}
\newcommand{\maxbi}{max-bigraphic\xspace}

\newcommand{\partition}{\textsc{Partition}\xspace}
\newcommand{\calF}{\mathcal{F}}

\title{Degree Realization by Bipartite Multigraphs%
\thanks{%
A preliminary version of this paper appeared in SIROCCO 2023~\cite{BBPR23}.
A preliminary version of one of the results of this paper appeared in SWAT 2022~\cite{bar2022realizing}.
This work was supported by US-Israel BSF grant 2022205.}}

\author[Amotz Bar-Noy et al.]{%
Amotz Bar-Noy\affiliationmark{1}
\and 
Toni B\"ohnlein\affiliationmark{2}
\and
David Peleg\affiliationmark{3}
\and 
Dror Rawitz\affiliationmark{4}%
\thanks{Part of the research was done while this author was on sabbatical at the 
department of computer science of Reykjavik University.}
}

\affiliation{%
The City University of New York (CUNY), New York, NY, USA.
\orcidlink{0009-0004-7021-9072}
\\
Huawei, Zurich, Switzerland.
\orcidlink{0009-0001-2152-022X}
\\
The Weizmann Institute of Science, Rehovot, Israel.
\orcidlink{0000-0003-1590-0506}
\\
Bar Ilan University, Ramat-Gan, Israel. 
\orcidlink{0000-0003-0323-6097}
}

\keywords{Degree Sequences, 
Graph Realization, 
Bipartite Graphs, 
Graphic Sequences, 
Bigraphic Sequences, 
Multigraph Realization.}

\begin{document}

\publicationdata{vol. 28:2}{2026}{7}{10.46298/dmtcs.15158}{2025-01-29; 2025-01-29; 2025-11-18; 2025-12-18}{2025-12-26}

\maketitle

\begin{abstract}
The problem of realizing a given degree sequence by a multigraph
can be thought of as a relaxation of the classical degree realization problem
(where the realizing graph is simple).
This paper concerns the case where the realizing multigraph is required
to be bipartite.

The problem of characterizing degree sequences that can be realized by
a bipartite (simple) graph has two variants. In the simpler one, termed BDR$^P$,
the partition of the degree sequence into two sides
is given as part of the input. A complete characterization for realizability
in this variant was given by Gale and Ryser over sixty years ago.
However, the variant where the partition is not given, termed BDR,
is still open.

For bipartite multigraph realizations, there are also two variants.
For BDR$^P$, where the partition is given as part of the input,
a complete characterization was known for determining whether 
there is a multigraph realization whose underlying graph is bipartite,
such that the \emph{maximum} number of copies of an edge is at most $r$.
We present a complete characterization for determining if there is a
bipartite multigraph realization such that the \emph{total} number of
excess edges is at most $t$.
We show that optimizing these two measures may lead to different realizations,
and that optimizing by one measure may increase the other substantially.
As for the variant BDR, where the partition is not given, we show that
determining whether a given (single) sequence admits a bipartite multigraph
realization is NP-hard. Moreover, we show that this hardness result extends 
to any graph family which is a sub-family of bipartite graphs and a super-family 
of paths.
On the positive side, we provide an algorithm
that computes optimal realizations for the case where the number of balanced 
partitions is polynomial, and present sufficient conditions for the existence
of bipartite multigraph realizations that depend only on the largest degree
of the sequence.
\end{abstract}




\section{Introduction}
\label{sec:intro}

\subsection{Background and Motivation}

\paragraph*{Degree realization.}
This paper concerns a classical network design problem known as the
\textsc{Graphic Degree Realization} problem (GDR).
The number of neighbors or connections of a vertex in a graph is called its
\emph{degree}, and it provides information on its centrality and importance.
For the entire graph, the sequence of vertex-degrees is a significant
characteristic which has been studied for over sixty years.
The graphic degree realization problem asks if a given non-increasing sequence of positive
integers $d = (d_1, ..., d_n)$ is \emph{graphic}, i.e., if it is the sequence of vertex-degrees of some graph. 
Erd\"os and Gallai~\cite{EG60} gave a characterization
for graphic sequences, though not a method for finding a \emph{realizing} graph.
Havel and Hakimi~\cite{havel55,hakimi62} proposed an algorithm that either
generates a realizing graph or proves that the sequence is not graphic.
Degree realization problems have found several interesting applications, 
most notably in network design, 
and also in the study of 
social networks~\cite{BD11,Cloteaux16,EGT14,MV02}, chemical networks~\cite{TatsuyaN13}, 
and network evolution~\cite{KMCET21}.


\paragraph*{Relaxed degree realization by multigraphs.}
An interesting direction in the study of realization problems involves
\emph{relaxed} (or approximate) realizations (cf. \cite{bar2021relaxed}). 
Such realizations are well-motivated by applications in two wider
contexts.
In scientific contexts, a given sequence may represent (noisy) data resulting
from an experiment, and the goal is to find a model that fits the data.
In such situations, it may happen that \emph{no} graph fits the input
degree sequence exactly, and consequently it may be necessary to search for
the graph ``closest'' to the given sequence.
In an engineering context, a given degree sequence constitutes constraints
for the design of a network. It might happen that satisfying all of the
desired constraints \emph{simultaneously} is not feasible, or causes
other issues, e.g., unreasonably increasing the costs.  
In such cases, relaxed solutions bypassing the problem may be relevant.


In the current paper we focus on a specific type of relaxed realizations
where the graph is allowed to have parallel edges, namely, the realization
may be a \emph{multigraph.}
It is easy to verify that if \emph{(multiple) self-loops} are allowed,
then \emph{every} sequence $d = (d_1, \ldots, d_n)$ whose sum $\sum_i d_i$
is even has a realization by a multigraph. Hence, we focus on the case where
self-loops are not allowed.

The problem of degree realization by multigraphs has been studied
in the past as well.
Owens and Trent~\cite{owens1967determining} gave a condition for the existence of a multigraph realization.
Will and Hulett~\cite{WillHulett04} studied the problem of finding a multigraph realization of a given sequence such that the underlying graph of the realization contains as few edges as possible. They proved that such a realization is composed of components, each of which is either a tree or a tree with a single odd cycle.
Hulett, Will, and Woeginger~\cite{HWW08} showed that this problem is strongly NP-hard.


\paragraph*{Degree realization by bipartite graphs.}
The \textsc{Bigraphic Degree Realization} problem (BDR) is a natural variant
of GDR, 
where the realizing graph is required to be bipartite.
The problem has a sub-variant, denoted BDR$^P$, in which \emph{two}
sequences are given as input, representing the vertex-degree sequences
of the two sides of a bipartite realizing graph.
(In contrast, in the general problem, a \emph{single} sequence is given
as input, and the goal is to find a realizing bipartite graph
based on some partition of the given sequence.) 
%
BDR$^P$ 
was solved by Gale and Ryser~\cite{gale1957theorem,ryser1957combinatorial} even before
Erd\"os and Gallai's characterization of graphic sequences.
However, the general problem -- mentioned as an open problem over forty years
ago~\cite{Rao81} -- remains unsolved today.

A (non-increasing) sequence of integers $d = (d_1, ..., d_n)$ can only be \emph{bigraphic}, i.e.,
the vertex-degree sequence of a bipartite graph, if it can be partitioned
into two sub-sequences or \emph{blocks} of equal total sum.
The latter problem is known as the \emph{partition problem} and it is solvable in polynomial time assuming that $d_1 < n$ 
(which is a necessary condition for $d$ to be bigraphic). Yet, BDR bears two obstacles. 
First, a sequence may have several partitions of which some are bigraphic and others are not. 
Second, the number of partitions may be exponentially large in $n$. 
Recent attempts on the BDR problem (see~\cite{bar2022realizing,BBPR25-hl,BBPR25-equal}) 
try to identify a small set of partitions, which are suitable to decide BDR for the whole sequence.
Each partition in the small set is tested using the Gale-Ryser characterization.
In case all of them fail the test, it is conjectured that no partition of the sequence is bigraphic.
The conjecture was shown to be true in case there exists a special partition that (perfectly) splits the degrees into small and large ones. 

Paralleling the above discussion concerning  relaxed degree realizations
by \emph{general} multigraphs,
one may look for relaxed degree realizations by \emph{bipartite multigraphs}.
This question is our main interest in the current paper.


\subsection{Our Contribution}

In this paper, we consider the problem of finding relaxed
\emph{bipartite multigraph} realizations for a given degree sequence
or a given partition. 
That is, the relaxed realizations must fulfill the degree constraints exactly
but are allowed to have parallel edges.
(Self-loops are disallowed.)

To evaluate the quality of a realization by a multigraph, we use two measures:
\begin{compactenum}[(i)]
\item The \emph{total multiplicity} of the multigraph, i.e., the number of parallel edges. 
\item The \emph{maximum multiplicity} of the multigraph, i.e., 
        the maximum number of edges between any two of its vertices.
\end{compactenum}
As shown later, these measures are non-equivalent, in the sense that
there are examples for sequences where realizations optimizing one measure
are sub-optimal in the other, and vice-versa.

\Cref{sec:prelim} introduces formally the basic notions and measures under study.
For relaxed realizations by \emph{general} multigraphs, 
it follows from the characterizations given, respectively, by 
Owens and Trent~\cite{owens1967determining}
and 
Chungphaisan~\cite{chungphaisan1974conditions},
how to optimize the two measurements and find the respective
multigraph realizations. 
For relaxed realization by \emph{bipartite} multigraphs,
finding a realization for BDR$^P$ (the given partition variant) that minimizes
the \emph{maximum} multiplicity follows from the characterization presented
by Berge~\cite{miller2013reduced}.

In \Cref{sec:bitotmult} we provide a characterization for bipartite multigraphs 
based on a given partition (BDR$^P$). 
More specifically, we present results on multigraph realizations
with bounded \emph{total} multiplicity for BDR$^P$.

In \Cref{sec:gap} we show that optimizing total multiplicity and 
maximum multiplicity may lead to different realizations. Moreover, 
optimizing by one measure may increase the other substantially.

One necessary condition for a (non-increasing) sequence $d = (d_1, \ldots, d_n)$ to be bigraphic
is that it can be \emph{partitioned}.
If $d_1 < n$, this problem can be decided in polynomial time. 
However, for a multigraph realization to exist, the inequality $d_1 < n$
is not a necessary condition, and it turns out that BDR is NP-hard.
We review this matter in greater detail in Section~\ref{sec:single-sequence-hardness} 
and show that this hardness results extends to any graph family which is a sub-family of bipartite graphs and a super-family of paths.
We discuss an output sensitive algorithm to generate all partitions of a given sequence.
In case the number of partitions of a sequence is small, the algorithm
allows us to find optimal realizations with respect to both criteria. 

In Section~\ref{sec:single-sequence-small-degree}, we discuss sufficient 
conditions for the existence of approximate bipartite realizations that depend
only on the largest degree of a given sequence.

\section{Preliminaries}
\label{sec:prelim}

Let $d = (d_1, d_2, \ldots, d_n)$ be a sequence of positive
integers in non-increasing order. (All sequences that we consider 
are assumed to be of positive integers and in a non-increasing order.)
The \emph{volume} of $d$ is $\vol{d} = \sum_{i=1}^{n} d_i$.
For a graph $G$, denote the sequence of its vertex-degrees by $\deg(G)$.
Sequence $d$ is \emph{graphic} if there is a graph $G$ such that $\deg(G) = d$.
We say that $G$ is a \emph{realization} of $d$.
Note that every realization of $d$ has $m = \vol{d}/2$ edges.
Consequently, a graphic sequence must have even volume. 
In turn, we call a sequence of positive integers with even volume a \emph{degree sequence}.
We use the operator $\circ$ to define $d \circ d'$ as the concatenation
of two degree sequences $d$ and $d'$ (in non-increasing order).


\subsection{Multigraphs as Approximate Realizations}

Let $H = (V,E)$ be a multigraph without loops.  In this case, $E$ is a
multiset.
Denote by $E_H(v,u)$ the multiset of edges connecting $v,u \in V$. 
If $\abs{E_H(v,u)}>1$, we say that the edge $(v,u)$ has $\abs{E_H(v,u)}-1$
\emph{excess} copies.
Let $E'$ be the set that is obtained by removing excess edges from $E$.
The graph $G = (V,E')$ is called the \emph{underlying graph} of $H$.

We view multigraphs as \emph{approximate} realizations of sequences that are not graphic.
Owens and Trent~\cite{owens1967determining} gave a condition for the existence of a multigraph realization.
\begin{theorem}[Owens and Trent~\cite{owens1967determining}]
A degree sequence $d$ can be realized by a multigraph if and only if $d_1 \leq \sum_{i=2}^n d_i$. 
\end{theorem}
To measure the quality of an approximate realization we introduce two metrics.
First, the \emph{maximum multiplicity} of a multigraph $H$ is the maximum number of  
copies of an edge, namely
\[
\maxmult(H) \triangleq \max_{(v,w) \in E } (\abs{E_H(v,w)})
~,
\]
and for a sequence $d$ define
\[
\maxmult(d) \triangleq \min\{ \maxmult(H) \mid H \text{ realizes }  d \}
~.
\]
We say that
a sequence $d$ is \emph{$r$-\maxg} if $\maxmult(d) \leq r$, for a positive integer $r$.

Second, the \emph{total multiplicity} of a multigraph $H$ is the total number of 
excess copies, namely
\[
\totmult(H) 
\triangleq \sum_{(v,w) \in E } (\abs{E_H(v,w)} - 1)
= \abs{E} - \abs{E'}
~,
\]
where $E'$ is the edge set of the underlying graph of $H$.
For a sequence $d$ define
\[
\totmult(d) \triangleq \min\{ \totmult(H) \mid H \text{ realizes }  d \}
~.
\]
We say 
a sequence $d$ is \emph{$t$-\totg} if $\totmult(d) \leq t$, for a positive integer $t$.



\subsection{General Multigraphs}

Given a degree sequence $d$, our goal is to compute $\maxmult(d)$ and $\totmult(d)$.

We note that the best realization in terms of maximum multiplicity is not necessarily the same
as the best one in terms of total multiplicity. See more on this issue in \Cref{sec:gap}.

Next, we iterate the characterization of Erd\"os and Gallai~\cite{EG60}
for graphic sequence.

\begin{theorem}[Erd\"os-Gallai~\cite{EG60}]
\label{thm:graphic}	
A degree sequence $\seqd$ is graphic if and only if,  
for $\ell = 1, \ldots, n$,
\begin{equation}
\label{eq:EG}
\displaystyle \sum_{i=1}^\ell d_i \le \ell(\ell-1) + \sum_{i= \ell + 1}^n \min\{\ell,d_i\}
~.
\end{equation}
\end{theorem}

Theorem~\ref{thm:graphic} implies an $\mathcal{O}(n)$ algorithm to verify whether a sequence is graphic. 
Chungphaisan~\cite{chungphaisan1974conditions} extended the above characterization 
to multigraphs with bounded maximum multiplicity as follows.

\begin{theorem}[Chungphaisan~\cite{chungphaisan1974conditions}] 
\label{thm:r-graphic}	
Let $r$ be a positive integer. A degree sequence $d$ is $r$-\maxg 
if and only if, for $\ell = 1, \ldots, n$,
\begin{equation}
\label{eq:r-graphic}
\displaystyle \sum_{i=1}^\ell d_i \leq r \ell(\ell-1)  + \sum_{i= \ell + 1}^n \min\{r \ell,d_i\}
~.
\end{equation}
\end{theorem}

Notice the similarity to the Erd\"os-Gallai equations. 
Moreover, since $\maxmult(d) \leq d_1$,
it follows that $\maxmult(d)$ can be computed in $\mathcal{O}(n \cdot \log(d_1) )$.

The problem of finding a multigraph realization with low total multiplicity 
was solved by Owens and Trent~\cite{owens1967determining}.
They showed that the minimum total multiplicity is equal to the 
minimum number of degree $2$ vertices that should be added to make 
the sequence graphic. We provide a simpler proof of their result.

\begin{theorem}[Owens and Trent~\cite{owens1967determining}]
\label{thm:totmult}
Let $d$ be a degree sequence such that $d_1 \leq \sum_{i=2}^n d_i$. 
Then, $d$ is $t$-\totg if and only if $\seqd \circ 2^t$ is graphic.
%
\end{theorem}
\begin{proof}
Let $\seqd$ be a degree sequence such that $d_1 \leq \sum_{i=2}^n d_i$. 
%
First, assume that $d$ 
can be realized by a multigraph $H$ with $\totmult(H)\le t$.
Let $F$ be the set of excess edges in $H$.
Construct a simple graph $G$ by replacing each edge $f = (x,y) \in F$
with two edges $(x,v_f)$ and $(y,v_f)$, where $v_f$ is a new vertex.
Clearly, this does not change the degrees of $x$ and $y$ and adds a vertex
$v_f$ of degree $2$. Hence the degree sequence of $G$ is $d \circ 2^{|F|}$.
Also, $G$ is simple.
If $|F| < t$, then one may replace any edge in $G$ with a path containing
$t-|F|$ edges, yielding a graph with degree sequence $d\circ 2^t$.

Conversely, suppose the sequence $\seqd \circ 2^t$ is graphic. 
Let $G$ be a simple graph that realizes the sequence.
Pick a degree $2$ vertex $v$ with neighbors $x$ and $y$, replace the edges
$(v,x)$ and $(v,y)$ with the edge $(x,y)$, and remove $v$ from $G$.
This transformation eliminates one degree 2 vertex from $G$
without changing the remaining degrees.
But it may increase the number of
excess edges by one (if the edge $(x,y)$ already exists in $G$).
Performing this operation for $t$ times, we obtain a multigraph $H$ with
$\totmult(H) \leq t$ and degree sequence $d$. 
\end{proof}

The next corollary follows readily with Theorems~\ref{thm:graphic} and \ref{thm:totmult}.
\begin{corollary}
\label{coro:totmultEG-pre}
Let $t$ be a positive integer, and let $d' = d \circ 2^t$. 
A sequence $d$ is $t$-\totg if and only if, for $\ell = 1, \ldots, n+t$,  
\begin{equation}
\label{eq:totmultEG}
\sum_{i=1}^{\ell} d'_i 
\leq \ell (\ell-1)  + \sum_{i = \ell + 1}^{n+t} \min\{\ell,d'_i\} 
~.
\end{equation}
\end{corollary}

%
Owens and Trent~\cite{owens1967determining} implicitly suggest to compute 
$\totmult(d)$ by computing the minimum $t$ such that $d \circ 2^t$ is graphic. 
Using binary search would lead to a running time of 
$\mathcal{O}(n \cdot \log(\totmult(d)))$. 

Several authors~\cite{TV03,ZZ92} noticed that the equations of Theorem~\ref{thm:graphic}
are not minimal.
For a degree sequence $d$ where $d_1 > 1$, let $\bidx{d} = max\{ i ~|~ d_i > i \}$.
If Equation~\eqref{eq:EG} holds for the index $\ell = \bidx{d}$, then it holds for index $\ell+1$. 
To see this, consider the equations for the two indices and compare the change in the 
left hand side (LHS) and right hand side (RHS).
Observe that the RHS increases at least by $(\ell+1) \cdot \ell - \ell \cdot (\ell-1) = 2\ell$ 
while the LHS only increases by $d_{\ell+1} \leq \ell$.   
It follows that Equation~\eqref{eq:EG} does not have to be checked for indices $\ell > \bidx{d}$.
If $d_1 = 1$, we define $\bidx{d} = 0$. Note that in this case $d$ is realized by a matching graph.

\begin{observation}[\cite{TV03,ZZ92}] 
\label{obs:graphic-box}	
A degree sequence $d$ is graphic 
if and only if, for $\ell = 1, \ldots, \bidx{d}$,
\begin{equation}
\label{eq:graphic-box}
\displaystyle \sum_{i=1}^\ell d_i \leq \ell(\ell-1)  + \sum_{i= \ell + 1}^n \min\{\ell,d_i\}
~.
\end{equation}
\end{observation}

On a side note,  it is also known that only up to $k$ many equations have to be checked where $k$ is the number of different degrees of a sequence (cf.~\cite{miller2013reduced,TV03,ZZ92}).

Observation~\ref{obs:graphic-box} helps to simplify Corollary~\ref{coro:totmultEG-pre}.

\begin{corollary}
\label{coro:totmultEG}
Let $t$ be a positive integer. A degree sequence $d$ 
is $t$-\totg 
if and only if,  for $\ell = 1, \ldots, \bidx{d}$, 
\begin{equation}
\label{eq:totmultEG-box}
\sum_{i=1}^{\ell} d_i 
\leq \ell (\ell-1)  + \sum_{i= \ell + 1}^n \min\{\ell,d_i\} + t \cdot \min\set{\ell,2}.
\end{equation}
\end{corollary}
\begin{proof}
Let $d$ and $t$ be as in the corollary.
In case $d_1 = 1$, the sequence $d$ is graphic, i.e., it is $t$-\totg for any positive integer $t$.

Hence, assume that $d_1 > 1$. Also, let $d' = \seqd \circ 2^t$.
One can verify that Equations~\eqref{eq:totmultEG-box} are the (reduced) Erd\"os-Gallai 
inequalities of Observation~\ref{obs:graphic-box} for $d'$. 
Moreover, $\bidx{d} = \bidx{d'}$, and the claim follows. 
\end{proof}

Corollary~\ref{coro:totmultEG} implies a simple algorithm to compute $\totmult(d)$.
Let
\[
\Delta_\ell(d) = \sum_{i=1}^{\ell} d_i - (\ell (\ell-1)  + \sum_{i= \ell + 1}^n \min\{\ell,d_i\}),
\]
for $\ell = 1, \ldots, n$, be the Erd\"os-Gallai differences of a 
sequence $d$. 
Hence, $\totmult(d) = \max\{ \Delta_1, \Delta_{\text{max}}/2 \}$, 
where $\Delta_{\text{max}}(d) = \max_{2 \leq \ell \leq \bidx{d}} \Delta_{\ell}(d)$,
implying a $\mathcal{O}(n)$ algorithm to calculate $\totmult(d)$.


\subsection{Bipartite Multigraphs}

In this section, we start investigating 
whether a degree sequence has a bipartite realization, i.e., if it is \emph{bigraphic} or not.
Particularly, we are interested in multigraph realizations where the underlying graph is bipartite.

Let $d$ be a degree sequence such that $\vol{d} = 2m$ for some integer $m$.
A \emph{block} of $d$ is a subsequence $a$ such that $\vol{a} = m$.
Define the set of blocks as $B(d)$. 
%
For each $a \in B(d)$ there is a disjoint $b \in B(d)$ such that $d = a \circ b$.
We call such a pair $a,b \in B(d)$ a balanced \emph{partition} of $d$ since $\vol{a} = \vol{b}$.
Denote the set of all balanced partitions of $d$ by 
$\bp(d) = \set{ \set{a,b} ~ | ~ a,b \in B(d),~ a \circ b = d  }$.
We say a partition $(a,b) \in \bp(d)$ is \emph{bigraphic} if 
there is a bipartite realization $G = (A,B,E)$ of $d$ 
such that $\deg(A) = a$ and $\deg(B) = b$ are the vertex-degree sequences of $A$ and $B$, respectively.

Observe that, as in the case of general graphs, the best realization in terms of maximum
 multiplicity is not necessarily the same as the best one in terms of total multiplicity. 
See \Cref{sec:gap}.

Note that not every graphic sequence has a balanced partition.
Yet, if $d$ is bigraphic, then $\bp(d)$ is not empty. 
The Gale-Ryser theorem characterizes when a partition is bigraphic.  
\begin{theorem}[Gale-Ryser~\cite{gale1957theorem,ryser1957combinatorial}]
\label{thm:Gale-Ryser}
Let $d$ be a degree sequence and partition $(a,b) \in \bp(d)$ where
$a = (a_1,a_2, \ldots, a_p)$ and $b = (b_1,b_2, \ldots,b_q)$.
The partition $(a,b)$ is bigraphic if and only if, for $\ell = 1, \ldots, p$,
\begin{equation}
\label{eq:GR}
\displaystyle \sum_{i=1}^{\ell} a_i \leq \sum_{i=1}^{q} \min\{\ell, b_i\}
~.
\end{equation}
\end{theorem}
We point out that Theorem~\ref{thm:Gale-Ryser} does not characterize bigraphic degree sequences.
Indeed, if the partition is not specified, it is not known how to determine whether a graphic sequence is bigraphic or not.
There are sequences where some partitions are bigraphic while others are not. 
Moreover, $|\bp(d)|$ might be exponentially large in the input size $n$.

We turn back to approximate realizations by bipartite multigraphs.
A multigraph is bipartite if its underlying graph is bipartite.
Analogue to above, we use the maximum and total multiplicity to measure the quality of a realization. 
Naturally, let 
\[
\maxmultbi(d) 
\triangleq \min\{ \maxmult(H) \mid H \text{ is bipartite and realizes }  d \}
~.
\]
For a partition $(a,b) \in \bp(d)$, we define 
\[
\maxmultbi(a,b) 
\triangleq \min\{ \maxmult(H) \mid H = (A,B,E) \text{ s.t. } \deg(A) = a \text{ and } \deg(B) = b \}
~.
\]

Let $r$ be a positive integer.
If there is a bipartite multigraph $H = (A , B, E)$ 
where $\maxmult(H) \leq r$, we say that $d$ is $r$-\maxbi. 
Moreover, we say that the partition $(a,b) \in \bp(d)$, 
where $a = \deg(A)$ and $b = \deg(B)$, is $r$-\maxbi.
Miller~\cite{miller2013reduced} cites the following result of Berge characterizing $r$-\maxbi partitions.

\begin{theorem}[Berge~\cite{miller2013reduced}]
\label{thm:r-multi-bi-graph}	
Let $r$ be a positive integer.
Consider a degree sequence $d$ and a partition $(a,b) \in \bp(d)$,
where $a = (a_1, \ldots, a_p)$ and $b = (b_1, \ldots, b_q)$.
Then $(a,b)$ is $r$-\maxbi if and only if,
for $\ell = 1, \ldots, p$, 
\begin{equation}
\label{eq:r-multi-bi-graph}	
\displaystyle \sum_{i=1}^{\ell} a_i \leq \sum_{i= 1}^q \min\{\ell r ,b_i\}
~.
\end{equation}
\end{theorem}

Note the similarity to the Gale-Ryser theorem. Theorem~\ref{thm:r-multi-bi-graph} implies 
that $\maxmultbi(a,b)$ can be computed in $\mathcal{O}(n \cdot \log(d_1) )$ using binary search.

For the second approximation criterion, we bound the total multiplicity of a bipartite multigraph realization.
Define
\[
\totmultbi(d) 
\triangleq \min\{ \totmult(H) \mid H \text{ is bipartite and realizes }  d \}
~.
\]
Additionally, for a partition $(a,b) \in \bp(d)$, we define
\[
\totmultbi(a,b) 
\triangleq \min\{ \totmult(H) \mid H = (A,B,E) \text{ s.t. } \deg(A) = a \text{ and } \deg(B) = b \}
~.
\]

We present our results on determining $\totmultbi(a,b)$ in the next section. 
In Sections~\ref{sec:single-sequence-hardness} and \ref{sec:single-sequence-small-degree}, 
we consider $\maxmultbi(d)$ and $\totmultbi(d)$.


\section{Multigraph Realizations of Bi-sequences}
\label{sec:bitotmult}

In this section, we are interested in bipartite multigraph
realizations with low total multiplicity, assuming that we are given a sequence 
and a specific balanced partition.
First, we provide a characterization similar to \Cref{thm:totmult}
for bipartite multigraph realizations for a given partition.

\begin{theorem}
\label{thm:totmult-bipartite}
Let $d$ be a degree sequence and $t$ be a positive integer.
Then, $d$ is $t$-\totbi
if and only if 
there exists a partition $(a,b) \in \bp(\seqd)$ such that $(a \circ (1^t), b \circ (1^t))$ is bigraphic.
\end{theorem}
\begin{proof}
Let $\seqd$, $t$ be as in the theorem.
Assume that there is a bipartite multigraph $H = (L, R, E)$
with $\totmult(H) \leq t$. Hence, there is a partition $(a,b) \in \bp(d)$
where $\deg(L) = a$ and $\deg(R) = b$.
Let $F$ be the set of excess edges in $H$.
Construct a bipartite graph $G$ by
applying the following transformation. For every excess edge $(x,y) \in F$,
add a new vertex $x_e$ to $A$ and a new vertex $y_e$ to $B$,
and replace $(x,y)$ by the two edges $(x,y_e)$ and $(x_e,y)$.
Note that $x_e$ and $y_e$ are placed on opposite partitions of $G$.
Since there are $t$ excess edges, $G$ realizes $(a \circ 1^t, b \circ 1^t)$ 
without excess edges.
	
For the other direction, assume that there exists a partition
$(a,b) \in \bp(\seqd)$ such that $(a \circ (1^t), b \circ (1^t))$ is realized
by a bipartite graph $G = (L,R,E)$.
Let $x_1, \ldots, x_t$ and $y_1, \ldots, y_t$ be some vertices of degree one
in $L$ and $R$, respectively.
Also, for every $i$, let $y'_i$ (respectively, $x'_i$) be the only neighbor
of $x_i$ (resp., $y_i$).
Construct a bipartite multigraph $H$ by replacing the edges $(x_i,y'_i)$
and $(x'_i,y_i)$ with the edge $(x'_i,y'_i)$ and discarding the vertices
$x_i$ and $y_i$, for every $i$. Since this transformation may add up to
$t$ excess edges, we have that $\totmult(H) \leq t$. 
\end{proof}

The above characterization leads to extended Gale-Ryser conditions for total multiplicity.

\begin{theorem}
\label{thm:totmult-bi-graph}
Let $d$ be a degree sequence with partition $(a,b) \in \bp(\seqd)$,
where $a = (a_1, \ldots, a_p)$ and $b = (b_1, \ldots, b_q)$, 
and let $t$ be a positive integer. 
The partition $(a,b)$ is $t$-\totbi 
if and only if,
for all $\ell = 1,\ldots,p$,
\begin{equation}
\label{eq:totmult-bi-graph}	
\sum_{i=1}^{\ell} a_i \leq \sum_{i= 1}^q \min\{\ell ,b_i\} + t 
~.
\end{equation}
\end{theorem}
\begin{proof}
Consider $(a,b)$ and $t$ as in the theorem. One can verify that
the following equations are the Gale-Ryser conditions of
Theorem~\ref{thm:Gale-Ryser} for the partition $(a \circ (1^t), b \circ (1^t))$ of Theorem~\ref{thm:totmult-bipartite}:
For all $\ell = 1,\ldots,p$, 
\begin{equation}
\label{eq:t-totmult-essential}	
\sum_{i=1}^{\ell} a_i \leq \sum_{i= 1}^q \min\{\ell ,b_i\} + t
~,
\end{equation}
and for all $h = 1,\ldots,t$,
\begin{equation}
\label{eq:t-totmult-obsolete}
\sum_{i=1}^p a_i + h \leq \sum_{i= 1}^q \min\{p+h ,b_i\} + t
~.
\end{equation}
	
To finish the proof, we argue that \Cref{eq:t-totmult-obsolete} holds 
for any $h \in \set{0,\ldots,t}$ if \Cref{eq:t-totmult-essential} holds for $\ell = p$.
Recall that $\sum_{i= 1}^q \min\{ p + h ,b_i\} = \sum_{i= 1}^q b_i$ if $p + h \geq b_1$.
It follows that Equation~\eqref{eq:t-totmult-obsolete} holds for indices $h \geq b_1 - p$.
	
Observe that
$\sum_{i= 1}^q \min\{ p + h + 1 ,b_i\} - \sum_{i= 1}^q \min\{ p + h ,b_i\} \geq 1$
for $p + h < b_1$, i.e., the RHS of Equation~\eqref{eq:t-totmult-obsolete}
grows by at least $1$ when moving from index $p + h$ to index $p + h +1$. 
By assumption, Equation~\eqref{eq:t-totmult-essential} holds for $\ell = p$,
implying that Equation~\eqref{eq:t-totmult-obsolete} holds for $h=0$.
Since the LHS of Equation~\eqref{eq:t-totmult-obsolete} grows by $1$ exactly,
Equation~\eqref{eq:t-totmult-obsolete} holds for indices $h < b_1 - p$. 
\end{proof}

Given a degree sequence $d$ with partition $(a,b) \in \bp(d)$,
Theorem~\ref{thm:totmult-bi-graph} implies that 
\[
\totmultbi(a,b) = \max_{1\leq \ell \leq p} \paren{ \sum_{i=1}^{\ell} a_i - \sum_{i= 1}^q \min\{\ell ,b_i\} }
~.
\]
It follows that $\totmultbi(a,b)$ can be computed in time $\mathcal{O}(n)$.



\section{Total Multiplicity vs. Maximum Multiplicity}
\label{sec:gap}

In this section we show that the measures of total multiplicity and
maximum multiplicity sometimes display radically different behavior. 
Specifically, we show that there are sequences such that 
a realization that minimizes the total multiplicity may be far from 
achieving minimum maximum multiplicity, and vice versa.

First, we notice that by definition of $\totmult$ and $\totmultbi$,
in order to minimize the total multiplicity one needs to use a 
maximum number of edges, or to maximize the number of edges 
in the underlying graph.

\begin{observation}
\label{obs:tot}
Let $d$ be a sequence and let $H = (V,E)$ be a multigraph that realizes $d$.
Let $G' = (V,E')$ be the underlying graph of $H$. Then,
\begin{itemize}
\item $\totmult(H) = \totmult(d)$ if and only if $|E'|$ is maximized.
\item $\totmultbi(H) = \totmultbi(d)$ if and only if $|E'|$ is maximized.
\end{itemize}
\end{observation}


\subsection{General Graphs}

Let $n \geq 5$, and consider the sequence 
\[
\hd = \left( (2n-2)^2, (n-1)^{n-2} \right)
~.
\]
Observe that $\sum \hd = (n - 1)(n +2)$.
The following two lemmas show that a realization of $\hat{d}$ attaining minimum total
multiplicity is far from obtaining minimum maximum multiplicity, and vise versa.

\begin{lemma}
\label{lemma:totgap}
$\totmult(\hd) = n-1$, and
if $H$ realizes $\hd$ such that $\totmult(H) = \totmult(\hd)$, then $\maxmult(H) = n$.
\end{lemma}
\begin{proof}
Consider a mutligraph $H$ which is composed of a full graph and 
$n-1$ copies of the edge $(1,2)$. More formally, let $H = (V,E)$ 
be a multigraph where
\[\textstyle
E = \set{(i,j) \mid 1 \leq i < j \leq n} \uplus \biguplus_{t=1}^{n-1} \set{(1,2)}
~.
\]
Recall that $E$ is a multiset.
See example in \Cref{fig:generaltot}.

We have that $\deg(1) = \deg(2) = (n-1) + (n-1) = 2n - 2$, and 
$\deg(i) = (n-1)$, for $i > 2$, as required. Thus, $H$ realizes $\hd$.
Observe that $\abs{E'} = \binom{n}{2}$, hence by \Cref{obs:tot} we have that
\[\totmult(\hd) = \totmult(H) = (n-1)~.\]

Let $\bar{H}$ be a realization such that $\totmult(\bar{H}) = \totmult(\hd)$.
Hence, $\abs{\bar{E}'} = \binom{n}{2}$.
Consider a vertex $i$, such that $i > 2$. All edges touching $i$ must be used at least once.
Since $\hd_i = n - 1$, all edge adjacent to $i$ must be used exactly once.
It follows that all excess edges are connected to the vertices $\set{1,2}$.
It follows that $\bar{H} = H$. Hence, $H$ minimizes the maximum multiplicity, 
and $\maxmult(H) = n$.
\end{proof}

\begin{figure}[t]
\centering
\subfloat[Optimal \totmult\ realization $H_1$.]{
\label{fig:generaltot}
\centering
~~~~~
\begin{footnotesize}
\begin{tikzpicture}[scale=1.2,auto,fill=lightgray,
    fdot/.style={circle,draw=black,fill=black,inner sep=0pt,minimum size=5pt},
    gdot/.style={circle,draw=black,fill=lightgray,inner sep=0pt,minimum size=5pt},
    edot/.style={circle,draw=black,fill=white,inner sep=0pt,minimum size=5pt}]
  \node[fdot] (v1) at (0.8,3) {};
  \node[fdot] (v2) at (3.2,3) {};
  \node[edot] (v3) at (4,1.2) {};
  \node[edot] (v4) at (2,0) {};
  \node[edot] (v5) at (0,1.2) {};
%
  \path (v1) edge (v2);
  \path (v1) edge [bend left = 10] (v2);
  \path (v1) edge [bend left = 20] (v2);
  \path (v1) edge [bend left = 30] (v2);
  \path (v1) edge [bend left = 40] (v2);
  \draw (v1) -- (v3);
  \draw (v1) -- (v4);
  \draw (v1) -- (v5);
  \draw (v2) -- (v3);
  \draw (v2) -- (v4);
  \draw (v2) -- (v5);
  \draw (v3) -- (v4);
  \draw (v3) -- (v5);
  \draw (v4) -- (v5);
\end{tikzpicture}
\end{footnotesize}
~~~~~
}
\hspace{30pt}
\subfloat[Optimal \maxmult\ realization $H_2$.]{
\label{fig:generalmax}
\centering
~~~~~
\begin{footnotesize}
\begin{tikzpicture}[scale=1.2,auto,fill=lightgray,
    fdot/.style={circle,draw=black,fill=black,inner sep=0pt,minimum size=5pt},
    gdot/.style={circle,draw=black,fill=lightgray,inner sep=0pt,minimum size=5pt},
    edot/.style={circle,draw=black,fill=white,inner sep=0pt,minimum size=5pt}]
  \node[fdot] (v1) at (0.8,3) {};
  \node[fdot] (v2) at (3.2,3) {};
  \node[edot] (v3) at (4,1.2) {};
  \node[edot] (v4) at (2,0) {};
  \node[edot] (v5) at (0,1.2) {};  
  \path (v1) edge [bend left = 10] (v2);
  \path (v1) edge [bend right = 10] (v2);
  \path (v1) edge [bend left = 10] (v3);
  \path (v1) edge [bend right = 10] (v3);
  \path (v1) edge [bend left = 10] (v4);
  \path (v1) edge [bend right = 10] (v4);
  \path (v1) edge [bend left = 10] (v5);
  \path (v1) edge [bend right = 10] (v5);
  \path (v2) edge [bend left = 10] (v3);
  \path (v2) edge [bend right = 10] (v3);
  \path (v2) edge [bend left = 10] (v4);
  \path (v2) edge [bend right = 10] (v4);
  \path (v2) edge [bend left = 10] (v5);
  \path (v2) edge [bend right = 10] (v5);
\end{tikzpicture}
\end{footnotesize}
~~~~~
}
\caption{Optimal multigraph realizations for the sequence $\hat{d} = (8^2,4^3)$ ($n=5$).
On the left we have $\totmult(H_1) = 4$ and $\maxmult(H_1) = 5$, while
on the right we have $\totmult(H_2) = 7$ and $\maxmult(H_2) = 2$.}
\label{fig:general}
\end{figure}
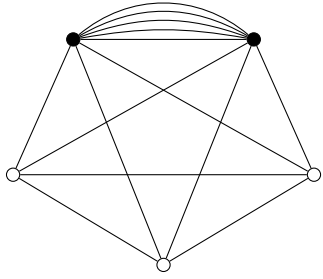
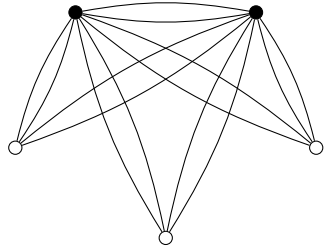

\begin{lemma}
\label{lemma:maxgap}
$\maxmult(\hd) = 2$, and if $H$ realizes $\hd$ such that $\maxmult(H) = \maxmult(\hd)$,
then $\totmult(H) \geq 2n - 3$.
\end{lemma}
\begin{proof}
Consider a vertex $1$ (or $2$). To minimize its load, its degree requirement 
should be distributed equally among the rest of the vertices. 
This leads to a realization $H$ in which each edge of $1$ and $2$ has two copies.
The degree requirements of the rest of the vertices are obtained by removing
a cycle from a complete graphs (this is the reason for requiring $n \geq 5$). 
Formally, $H = (V,E)$ is define as follows:
\[
E = \set{(1,i), (1,i) \mid i \geq 2} \uplus
      \set{(2,i), (2,i) \mid i \geq 3} \uplus
      \set{(i,j) \mid 2 < i, j \neq i+1 \text{ and } (i,j) \neq (3,n)} 
~.
\]
See example in \Cref{fig:generalmax}.

We have that $\deg(1)  = \deg(2) = 2(n-1)$, and $\deg(i)= 2+(n-1-2) = n-1$, 
for $i > 2$, as required. Thus, $H$ realizes $\hd$. 
Moreover, each edges has at most two copies, which means that $\maxmult(\hd)=2$.

Observe that an edge $(i,j)$, where $i, j  > 2$ has at most a single copy. 
Hence, $H$ minimizes the total multiplicity. 
In addition, $\totmult(H) = (n - 1) + (n - 2) = 2n - 3$. 
\end{proof}

\begin{corollary}
Let $n \geq 5$.
There exists a sequence $\hd$ of length $n$ such that $\totmult(H_2) - \totmult(H_1) = n-2$ 
and $\maxmult(H_1) - \maxmult(H_2) = n-2$, for any 
with two realizations $H_1$ and $H_2$ of $\hd$
such that $\totmult(H_1) = \totmult(\hd)$ and $\maxmult(H_2) = \maxmult(\hd)$.
\end{corollary}


\subsection{Bipartite Graphs}

Let $n$ be an even integer such that $n\geq 4$, and consider the sequence
\[
\td = \left( n^2, (n/2)^{n-2} \right)
~.
\]

\begin{lemma}
\label{lemma:bitotgap}
$\totmultbi(\td) = n/2$, and if $H$ realizes $\td$ such that $\totmultbi(H) = \totmultbi(\td)$,
then $\maxmultbi(H) \geq n/2+1$.
\end{lemma}
\begin{proof}
Let $a = b = (n, (\frac{n}{2})^{(n-2)/2})$. We construct a multigraph $H$ 
that realizes $(a,b)$, which consist of a complete bipartite graph and 
$n/2$ copies of the edge $(1,1)$.
Formally, $H = (A,B,E)$, where 
\[\textstyle
E = \set{(i,j) \mid 1\leq i,j \leq n/2}\uplus\biguplus_{t=1}^{n/2} \set{(1,1)}
~.
\]
See example in \Cref{fig:bitot}.

On both sides we have that $\deg(1) = \frac{n}{2} + \frac{n}{2} = n$, 
and $\deg(i) = \frac{n}{2}$, for $i > 1$, as required. 
Thus, $H$ realizes $(a,b)$.
Observe that $\abs{E'} = \frac{n}{2} \cdot \frac{n}{2}$, hence by
\Cref{obs:tot} we have that 
\[
\totmultbi(\td) = \totmultbi(a,b) = \totmultbi(H) = n/2
~.
\]

Let $\bar{H}$ be a realization such that $\totmultbi(\bar{H})=\totmultbi(a,b)$.
It follows that $\abs{\bar{E}'} = \frac{n}{2} \cdot \frac{n}{2}$.
Consider a vertex $i$, such that $i > 2$. All edges touching $i$ must be used at least once.
Since $\td_i = \frac{n}{2}$, all edges touching $i$ must be used exactly once.
It follows that all excess edges are connected to the vertices $\set{1,2}$.
Hence, $\bar{H} = H$. Also, $\maxmultbi(H) = n/2 + 1$.
\end{proof}

\begin{figure}[t]
\centering
\subfloat[Optimal $\totmultbi$ realization $H_1$.]{
\label{fig:bitot}
\centering
~ \hspace{5pt}
\begin{footnotesize}
\begin{tikzpicture}[scale=1.5,auto,fill=lightgray,
    fdot/.style={circle,draw=black,fill=black,inner sep=0pt,minimum size=5pt},
    gdot/.style={circle,draw=black,fill=lightgray,inner sep=0pt,minimum size=5pt},
    edot/.style={circle,draw=black,fill=white,inner sep=0pt,minimum size=5pt}]
  \node[fdot] (v1) at (0,2) {};
  \node[edot] (v3) at (0,1) {};
  \node[edot] (v5) at (0,0) {};
  \node[fdot] (v2) at (2,2) {};
  \node[edot] (v4) at (2,1) {};
  \node[edot] (v6) at (2,0) {};
%
  \path (v1) edge [bend left = 5] (v2);
  \path (v1) edge [bend left = 15] (v2);
  \path (v1) edge [bend right = 5] (v2);
  \path (v1) edge [bend right = 15] (v2);
  \draw (v1) -- (v4);
  \draw (v1) -- (v6);
  \draw (v3) -- (v2);
  \draw (v3) -- (v4);
  \draw (v3) -- (v6);
  \draw (v5) -- (v2);
  \draw (v5) -- (v4);
  \draw (v5) -- (v6);
\end{tikzpicture}
\end{footnotesize}
\hspace{5pt} ~
}
\hspace{5pt}
\subfloat[Optimal $\maxmultbi$ realization $H_2$.]{
\label{fig:biab}
\centering
~ \hspace{5pt} 
\begin{footnotesize}
\begin{tikzpicture}[scale=1.5,auto,fill=lightgray,
    fdot/.style={circle,draw=black,fill=black,inner sep=0pt,minimum size=5pt},
    gdot/.style={circle,draw=black,fill=lightgray,inner sep=0pt,minimum size=5pt},
    edot/.style={circle,draw=black,fill=white,inner sep=0pt,minimum size=5pt}]
  \node[fdot] (v1) at (0,2) {};
  \node[edot] (v3) at (0,1) {};
  \node[edot] (v5) at (0,0) {};
  \node[fdot] (v2) at (2,2) {};
  \node[edot] (v4) at (2,1) {};
  \node[edot] (v6) at (2,0) {};
  \path (v1) edge [bend left = 5] (v2);
  \path (v1) edge [bend right = 5] (v2);
  \path (v1) edge [bend left = 5] (v4);
  \path (v1) edge [bend right = 5] (v4);
  \path (v1) edge [bend left = 5] (v6);
  \path (v1) edge [bend right = 5] (v6);
  \path (v2) edge [bend left = 5] (v3);
  \path (v2) edge [bend right = 5] (v3);
  \path (v2) edge [bend left = 5] (v5);
  \path (v2) edge [bend right = 5] (v5);
  \path (v3) edge (v6);
  \path (v5) edge (v4);
\end{tikzpicture}
\end{footnotesize}
\hspace{5pt} ~
}
\hspace{5pt}
\subfloat[Optimal $\maxmultbi$ realization $H_3$.]{
\label{fig:bimax}
\centering
~ \hspace{5pt} 
\begin{footnotesize}
\begin{tikzpicture}[scale=1.5,auto,fill=lightgray,
    fdot/.style={circle,draw=black,fill=black,inner sep=0pt,minimum size=5pt},
    gdot/.style={circle,draw=black,fill=lightgray,inner sep=0pt,minimum size=5pt},
    edot/.style={circle,draw=black,fill=white,inner sep=0pt,minimum size=5pt}]
  \node[fdot] (v1) at (0,1.5) {};
  \node[fdot] (v2) at (0,0.5) {};
  \node[edot] (v3) at (2,2) {};
  \node[edot] (v4) at (2,1.33) {};
  \node[edot] (v5) at (2,0.66) {};
  \node[edot] (v6) at (2,0) {};
  \path (v1) edge [bend left = 5] (v3);
  \path (v1) edge [bend right = 5] (v3);
  \path (v1) edge [bend left = 5] (v4);
  \path (v1) edge [bend right = 5] (v4);
  \path (v1) edge (v5);
  \path (v1) edge (v6);
  \path (v2) edge (v3);
  \path (v2) edge (v4);
  \path (v2) edge [bend left = 5] (v5);
  \path (v2) edge [bend right = 5] (v5);
  \path (v2) edge [bend left = 5] (v6);
  \path (v2) edge [bend right = 5] (v6);
\end{tikzpicture}
\end{footnotesize}
\hspace{5pt} ~
}
\caption{Multigraph bipartite realizations for the sequence $\tilde{d} = (6^2,3^4)$.
On the left we have $\totmultbi(H_1) = 3$ and $\maxmultbi(H_1) = 4$;
In the center we have $\totmultbi(H_2) = 5$ and $\maxmultbi(H_2) = 2$;
On the right we have $\totmultbi(H_3) = 4$ and $\maxmultbi(H_3) = 2$.}
\label{fig:bipartite}
\end{figure}
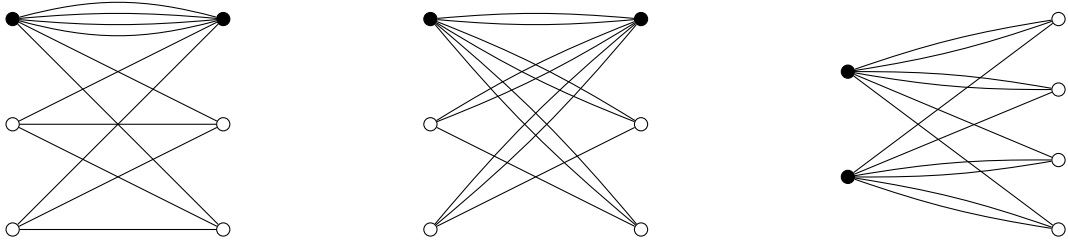

\begin{lemma}
\label{lemma:bimaxgap}
$\maxmultbi(\td) = 2$, and if $H$ realizes $\td$ such that $\maxmult(H) = \maxmult(\td)$,
then $\totmult(H) \geq n-2$.
\end{lemma}
\begin{proof}
There are two possible partitions for $\td$:
\begin{description}
\item{(P1)} $a = b = \left(n, (\frac{n}{2})^{(n-2)/2}\right)$, and
\item{(P2)} $a'=\left(n^2,(\frac{n}{2})^{n/2-3}\right)$ and
    $b'=\left((\frac{n}{2})^{n/2+1}\right)$.
\end{description}
First examine partition~(P1).
Consider a vertex $1 \in A$ (or $1 \in B$). To minimize its load, its degree 
requirement should be distributed equally among the rest of the vertices on 
the other side of the partition. This leads to a realization $H$, where vertex 
$1 \in A$ is connected to all the vertices in $B$ by two copies, while vertex 
$1 \in B$ is connected to all the vertices in $A$ by two copies. The 
requirement of the other vertices is obtained using a complete bipartite 
graph minus a perfect matching.
Hence, $H = (A,B,E)$, where 
\[
E = \set{(1,1), (1,1)} \uplus \set{(1,i), (1,i), (i,1), (i,1)\mid i \geq 2} \uplus
      \set{(i,j) \mid i,j \geq 2, j \neq i }
~.
\]
See example in \Cref{fig:biab}.

We have that $\deg(1) = 2(n/2) = n$, and $\deg(i) = 2 + (n/2 - 2) = n/2$, 
for $i > 2$, on both sides, as required. Thus, $H$ realizes $(a,b)$. Moreover, 
$\maxmultbi(a,b) = \maxmultbi(H) = 2$, since each edge has at most two copies.
Observe that an edge $(i,j)$, where $i, j  > 1$ has at most a single copy. 
Hence, $H$ minimizes the maximum multiplicity. 
In addition, $\totmult(H) = n/2 + n/2 - 1 = n - 1$. 

Next, consider partition (P2).
We construct a realization $H'' = (V'',E'')$ as follows:
\[ 
E'' = \set{(i,j) \mid i \leq 2 \text{ or } i > 2 \text{ and } j \neq i} 
      \uplus \set{(1,j) \mid j \leq n/2 - 1}
      \uplus \set{(2,j) \mid j \geq 3}~. 
\]
See example in \Cref{fig:bimax}.

On the left side, we have that $\deg(1) =\deg(2) = n/2+1 + n/2-1 = n$, 
and $\deg(i) = n/2 + 1 - 1 = n/2$, for $i > 3$.
On the right side, $\deg(j) = n/2-1 + 2$, for $i \in \set{3,\ldots,n/2-1}$, and
$\deg(j) = n/2 + 1$, for $i \in \set{1,2,n/2,n/2+1}$.
Thus, $H''$ realizes $(a',b')$. Moreover, 
$\maxmultbi(a',b') = \maxmultbi(H'') = 2$, since each edge has at most two copies.
Furthermore, $\totmult(H'') = 2(n/2 - 1) = n - 2$.
\end{proof}

\begin{corollary}
Let $n$ be an even integer such that $n \geq 4$.
There exists a sequence $\td$ of length $n$ such that 
$\totmultbi(H_2) - \totmultbi(H_1) = n/2 - 2$ and
$\maxmultbi(H_1) - \maxmultbi(H_2) = n/2 - 1$, 
for any 
with two realizations $H_1$ and $H_2$ of $\td$
such that $\totmultbi(H_1)=\totmultbi(\td)$ and $\maxmultbi(H_2)=\maxmultbi(\td)$.
\end{corollary}


\section{Bipartite Realization of a Single Sequence}
\label{sec:single-sequence-hardness}

In this section, we study the following question: given a degree sequence $d$, 
can it be realized as a multigraph whose underlying graph is bipartite?
Also, if there exists such a realization, we would like to find one which 
minimizes the maximum or the total multiplicity.


\subsection{Hardness Result}

Given a sequence and a balanced partition one may construct a bipartite 
multigraph realization by assigning edges in an arbitrary manner.

\begin{observation}
\label{obs:bi-sequence}
Let $d$ be a sequence and let $(\ell,r) \in \bp(d)$ be a partition
of $d$. Then, there exists a bipartite multigraph realization of $(\ell,r)$.
\end{observation}

It follows that deciding whether a degree sequence $d$ can it be 
realized as a multigraph whose underlying graph is bipartite is NP-hard.

\begin{theorem}
\label{thm:bipartite-hard}
Deciding if a degree sequence $d$ admits a bipartite multigraph realization is NP-hard.
\end{theorem}
\begin{proof}
We prove the theorem using a reduction from the \textsc{Partition} problem.
Recall that \textsc{Partition} contains all sequences
$(a_1,\ldots,a_n)$ such that there exists an index set $S \subseteq
[1,n]$ for which $\sum_{i \in S} a_i = \sum_{i \not\in S} a_i$
(see. e.g., \cite{GJ79}).
Observation~\ref{obs:bi-sequence} implies that $\seqd$ is a \textsc{Partition} 
instance if and only if $d$ admits a bipartite mulitgraph realization.
\end{proof}

\Cref{obs:bi-sequence} also implies a reduction from bipartite multigraph realization to \textsc{Partition}.
Since \textsc{Partition} admits a pseudo-polynomial time algorithm,
we have the following.

\begin{theorem}
\label{thm:bipartite-poly}
Deciding if a sequence $d$ admits a bipartite multigraph realization
can be done in pseudo-polynomial time.
\end{theorem}

We note that there may be an exponential number of balanced partitions of a sequence $d$ 
even if $d_1 < n$ (see, e.g.,~{\cite{BBPR25-hl})}.
In bipartite multigraph realization it is enough to find any balanced partition.
However, in BDR the partition should also satisfy the Gale-Ryser conditions.


Next, we show that deciding whether a given sequence has a multigraph realization whose 
underlying graph belongs to a graph family is hard for any family which is a subfamily of bipartite graphs
and a super family of paths.
We start be defining the following variant of \partition we refer to as $\partition'$.
A sequence of integers $b = (b_1,\ldots,b_n)$ is in $\partition'$ 
if and only if 
\begin{enumerate}
\item $n$ is even.
\item There exists $B > 0$ such that $b_i \geq 2B$, for every $i \in \set{1,\ldots,n-2}$, 
        $b_{n-1} = b_n = B$, and $\sum_{i = 1}^n b_i = (2n-1)B$.
\item There exists an index set $S$ such that $\sum_{i \in S} b_i = \sum_{i \not\in S} b_i$.
\end{enumerate}

\begin{observation}
\label{obs:sum}
Let $d$ be a sequence that satisfies the first two conditions of $\partition'$, and let $S \subseteq \set{1,\ldots,n-2}$.
Then, $\sum_{i \in S} b_i \leq B (2\abs{S} + 1)$.
\end{observation}
\begin{proof}
\[
\sum_{i \in S} b_i 
= \sum_{i=1}^n b_i - \sum_{i \not\in S, i<n-1} b_i - b_{n-1} - b_n
\leq (2n - 1) \cdot B  - (n - 2 - \abs{S}) \cdot 2B - 2B
=    B (2\abs{S} + 1)
~.
\]
\end{proof}

\begin{observation}
\label{obs:B}
Let $b \in \partition'$, and let $S$ be an index set such that $\sum_{i \in S} b_i = \sum_{i \not\in S} b_i$.
Then, $\abs{S} = n/2$ and $\abs{S \cap \set{n-1,n}} = 1$.
\end{observation}
\begin{proof}
Assume that $\abs{S} \leq n/2 - 1$, and let $S' = S \cap \set{1,\ldots,n-2}$ and $S'' = S \cap \set{n-1,n}$.
By \Cref{obs:sum} we have that
\[
\sum_{i \in S} b_i 
\leq B (\abs{S'} \cdot 2 + 1) + B \abs{S"}
\leq B (\abs{S} \cdot 2 + 1) 
\leq B (n - 2 + 1)
=    B (n - 1)
~.
\]
A contradiction.
A similar argument works for the case where $\abs{S} \geq n/2+1$.

Let $\abs{S} = n/2$ and assume that $\set{n-1,n} \subseteq S$. It follows that 
\[
\sum_{i \in S} b_i 
=    2B + \sum_{i \in S, i \leq n-2} b_i 
\leq 2B + (n/2-2)2B + B
=    (n - 1)B
~.
\]
A contradiction.
A similar argument works for the case where $\set{n-1,n} \cap S = \emptyset$.
\end{proof}

We show that this variant of \partition is NP-hard.

\begin{lemma}
$\partition'$ is NP-hard.
\end{lemma}
\begin{proof}
We prove the theorem by a reduction from \textsc{Partition}.
Given a sequence $a = (a_1,\ldots,a_n)$, where $B = \sum_i a_i$, we
construct the following degree sequence $b$ as follows:
\[
b_j = \begin{cases}
2B + a_j & j \in \set{1,\ldots,n}, \\
2B         & j \in \set{n+1,\ldots,2n}, \\
B           & j \in \set{2n+1,2n+2}.
\end{cases}
\]
The length of $b$ is $2n+2$ which is even.
Observe that $b_i \geq 2B$, for every $i \in \set{1,\ldots,2n}$, 
and $b_{2n+1} = b_{2n+2} = B$. Also, 
\[
\sum_{i = 1}^{2n+2} b_i 
= 4nB + \sum_i a_i + 2B 
= (4n + 3)B 
= (2(2n+2) - 1)B
~.
\]
Hence it remains to show that $a \in \partition$ if and only if there exists 
an index set $S$ such that $\sum_{i \in S} b_i = \sum_{i \not\in S} b_i$ 
and $\abs{S} = n+1$.

Suppose that $a \in \partition$ and let $T$ be an index set such that 
$\sum_{i \in T} a_i = \sum_{i \not\in T} a_i$. 
Let $S = T \cup \set{i + n : i \not\in T} \cup \set{2n+1}$.
Observe that $\abs{S} = n + 1 $ and
\[
\sum_{i \in S} b_i 
= \sum_{i \in T} (2B+a_i) + (n-\abs{T}) 2B + B
= (2n+1)B + B/2
~,
\]
As required.

On the other hand, assume that $b \in \partition'$ and let $S$ be an index 
set such that $\sum_{i \in S} b_i = \sum_{i \not\in S} b_i$ and $\abs{S} = n+1$.
By \Cref{obs:B} we may assume, without loss of generality, that $2n+1 \in S$ and 
$2n+2 \not\in S$.
Let $T = S \cap \set{1,\ldots,n}$. We have that
\[
\sum_{i \in T} a_i 
= \sum_{i \in T} (b_i - 2B)
= \sum_{i \in S} b_i - \abs{T} 2B - (n-\abs{T})2B - B
= (2n+1)B + B/2 - 2nB - B
= B/2
~.
\]
\end{proof}

It is said that a sequence $b$ has a \emph{sound} permutation if the following conditions hold:
\begin{enumerate}
\item $\sum_{i=1}^{n} (-1)^i b_{\pi(i)} = 0$.
  \label{sound:eq} 
  
\item $\sum_{i=1}^{k} (-1)^{i-1} b_{\pi(k-i+1)} > 0$, for all $k < n$.
  \label{sound:2}
\end{enumerate}

Next, we show that a sequence $b \in \partition'$ has a sound permutation.

\begin{lemma}
\label{lemma:sound}
If $b \in \partition'$, then $b$ admits a sound permutation.
\end{lemma}
\begin{proof}
Let $S$ be an index set such that $\sum_{i \in S} b_i = \sum_{i \not\in S} b_i$ and $\abs{S} = n/2$.
By \Cref{obs:B} we may assume, without loss of generality, that $n-1 \in S$ and 
$n \not\in S$.
We define permutation $\pi$ as follows.
First, let $\pi(1) = n-1$, and $\pi(n) = n$. Also, assign the remaining $n/2 - 1$ members of $S$ to odd 
indices. The remaining $n/2 - 1$ non-members of $S$ are assigned to even indices.

Condition~\ref{sound:eq} is satisfied, since
\[
\sum_{i : \pi(i) \text{ is odd}} b_i 
= \sum_{i \in S} b_i 
= \sum_{i \not\in S} b_i 
= \sum_{i : \pi(i) \text{ is even}} b_i 
~.
\]

It remains to prove that Condition~\ref{sound:2} is satisfied.
If $k$ is odd, we have that
\[
\sum_{i=1}^{k} (-1)^{i-1} b_{\pi(k-i+1)}
=      \sum_{j = 1}^{(k+1)/2} \!\! b_{2j-1} - \!\! \sum_{j = 1}^{(k-1)/2} \!\! b_{2j} 
\geq [B + (k-1)/2 \cdot 2B] - [(k-1)/2 \cdot 2B + B/2]
\geq B/2
~.
\]
Otherwise, 
\[
\sum_{i=1}^{k} (-1)^{i-1} b_{\pi(k-i+1)}
=      -\sum_{j = 1}^{k/2} b_{2j-1} + \sum_{j = 1}^{k/2} b_{2j} 
\geq -[B + (k/2 - 1) \cdot 2B + B/2] + [k/2 \cdot 2B] 
\geq B/2
~.
\]
The lemma follows.
\end{proof}

We are now ready for the Hardness result regarding bipartite multigraphs.

\begin{theorem}
\label{thm:hard}
Let $\calF$ be a family of bipartite graphs which contains all paths.
It is NP-hard to decide if a degree sequence $d$ admits a multigraph realization
whose underlying graph is in $\calF$.
\end{theorem}
\begin{proof}
We prove the theorem using a reduction $f$ from \partition'.
The reduction is as follows: $d = f(b) = b$, 
if $n$ is even and there exists $B > 0$ such that $b_i \geq 2B$, for every $i \in \set{1,\ldots,n-2}$, 
 $b_{n-1} = b_n = B$, and $\sum_{i = 1}^n b_i = (2n-1)B$.
Otherwise, $d = f(b) = d_0$, where $d_0$ is a sequence that cannot be realized.
Hence, we need to show that $b \in \partition'$ if and only if 
$d$ is realizable using an underlying graph from $\calF$.

First assume that $b \in \partition'$. In this case, $d = b$.
Hence, there exists a sound permutation $\pi$ for $d$.
Define the following multigraph $H$ whose underlying graph $G$ is a path.
The number of edges between $v_{\pi(k)}$ and $v_{\pi(k+1)}$ is 
\[
\sum_{i=1}^{k} (-1)^{i-1} d_{\pi(k-i+1)}
~.
\]
Since $\pi$ is sound, these numbers are positive.
It is not hard to verify that $H$ realizes $d$.

Now assume that $d$ realizable by a graph $G$ from $\calF$. It follows that $d = f(b) = b$.
It follows that $n$ is even and there exists $B > 0$ such that $b_i \geq 2B$, 
for every $i \in \set{1,\ldots,n-2}$, $b_{n-1} = b_n = B$, and $\sum_{i = 1}^n b_i = (2n-1)B$.
Since $G \in \calF$, it is bipartite, and thus there are two partitions $L$ and $R$, such that
\[
\sum_{j \in L} d_j = \sum_{j \in R} d_j
~.
\]
Hence, $b \in \partition'$.
\end{proof}

\begin{corollary}
\label{cor:hard}
It is NP-hard to decide if a degree sequence $d$ admits a multigraph realization
whose underlying graph is 
a path, a caterpillar, a bounded-degree tree, a tree, a forest, and a connected bipartite graph.
\end{corollary}


\subsection{Computing all Balanced Partitions of a Degree Sequence}

We describe an algorithm that given a degree sequence $d$, computes 
all balanced partitions of $d$. The algorithm relies on the self-reducibility 
of the \textsc{Subset-Sum} problem. Recall that in \textsc{Subset-Sum}
the input is a sequence of numbers $(a_1,\ldots,a_n)$ and an additional number $t$, 
and the question is whether there is a subset $S$ such that $\sum_{i \in S} a_i = t$.
Let Subset-Sum-DP be a dynamic programming algorithm for \textsc{Subset-Sum}
whose running time is denoted by $T_{DP}(a,t)$ (see, e.g., \cite{CLRS09}).
The running time of the dynamic programming algorithm can be bounded by 
$O(n \cdot \min\set{\sum a, \abs{\bp(a)}})$.

In this section we abuse notation by presenting a sequence $d$ as a sequence of $q$ 
blocks, namely $d = (d_1^{n_1}, d_2^{n_2},\ldots,d_q^{n_q})$.
The algorithm for computing all balanced partitions of a sequence $\seqd$ is recursive, 
and it works as follows. 
The input is a suffix of $d$, i.e., $(d_k^{n_k},\ldots,d_q^{n_q})$, represented by 
$d$ and $k$, and a partition $(L,R)$ of the prefix $(d_1^{n_1},\ldots,d_{k-1}^{n_{k-1}})$.
If the current suffix is empty, then it checks whether the current partition is balanced, 
and if it is balanced, then the partitioned is returned. 
Otherwise, it checks whether the current partition can be completed to a balanced partition. 
If the answer is YES, then the algorithm is invoked for the $n_k+1$ options of adding the 
$n_k$ copies of $d_k$ to $(L,R)$.
The initial call is $(d,1,\emptyset,\emptyset)$.

\begin{algorithm}[ht]
\caption{Partitions$(d,k,L,R)$}
\label{alg:partitions}
	\eIf {$k = q+1$}{
		\lIf {$(L,R) \in \bp(d)$}{
			\Return $\set{(L,R)}$
		}
	}{
		\lIf {$\text{Subset-Sum-DP}((d_k^{n_k},\ldots,d_q^{n_q}),\sum L - \sum R) = \text{NO}$}{
			\Return $\emptyset$
		}
	}
	$\cP \gets \emptyset$ \\
	\For {$i=0$ to $n_k$}{
		$L' \gets L \circ (d_k^i)$ \\
		$R' \gets L \circ (d_k^{n_k-i})$ \\
		$\cP \gets  \cP \cup \text{Partitions}(d,k+1,L',R')$ 
	}	
	\Return $\cP$
\end{algorithm}

\begin{lemma}
Algorithm \text{Partitions} returns all balanced partitions of $d$.
\end{lemma}
\begin{proof}
Observe that each recursive call of the algorithm corresponds to a partition of 
a prefix of $d$.
We prove that, given a prefix partition, the algorithm returns all of its balanced completions.

At the recursion base, if $(L,R)$ is a partition of $d$, 
then it is returned if and only if  $(L,R) \in \bp(d)$.
For the inductive step, let $(L,R)$ be a partition of the prefix 
$(d_1^{n_1},\ldots,d_{k-1}^{n_{k-1}})$.
If $(L,R)$ gets a NO from Subset-Sum-DP, then it cannot be completed to a balanced partition, 
and indeed no partition that corresponds to the prefix $(L,R)$ is returned.
If $(L,R)$ gets a YES, then all possible partitions of $(d_1^{n_1},\ldots,d_k^{n_k})$ are checked.
By the inductive hypothesis the algorithm returns all balanced partitions that complete 
$(d_1^{n_1},\ldots,d_{k-1}^{n_{k-1}})$.
\end{proof}

The complexity of Algorithm Partitions is dominated by the total time
spent on the invocations of Subset-Sum-DP, therefore we need to bound  the
number of invocations of Subset-Sum-DP. 
More specifically, we show the following bound.

\begin{lemma}
Algorithm \text{Partition} invokes Subset-Sum-DP at most $2n \cdot \abs{\bp(d)}$ times.
\end{lemma}
\begin{proof}
Let us illustrate the recursive execution of the algorithm on $d$ by a
computation tree $T$ consisting of $q+1$ levels.
Each node in the tree is labeled by a triple $(L,R,A)$, where $A \in \{YES,NO\}$,
The first two entries in the label corresponds to the prefix partition $(L,R)$ in the invocation,
and the third corresponds to whether the partition can be completed to a balanced partition.
Note that each such node corresponds to a single invocation of Subset-Sum-DP (or alternatively 
checking whether $\sum L = \sum R$ in level $q+1$). It follows that the number of invocations 
of Subset-Sum-DP is bounded by the size of the computation tree, not including level $q+1$.

We refer to a node as a YES-node (NO-node) if its label end with a YES (NO).
Observe that all NO-nodes are leaves. On the other hand, there may be internal YES-nodes.
If a YES-node is a leaf, then it corresponds to a balanced partition $(L,R)$.
Clearly, the number of YES-nodes in level $k+1$ of the tree is no less than the number of YES-nodes in level $k$.  
Moreover, a NO-node must have a YES-node as a sibling, hence 
the number of NO-nodes in level $k \leq q$ is at most $n_k$ times the number of YES-nodes in level $k$.
Adding it all up we get:
\[
\sum_{k=1}^q \abs{\bp(d)} (n_k +1)
=    (n + q) \abs{\bp(d)}
\leq 2n \cdot \abs{\bp(d)}
~.
\qedhere
\]
\end{proof}

Hence, the lemma allows us to get an upper bound on the time complexity of the algorithm.

\begin{corollary}
The running time of Algorithm \text{Partition} is $O(n^2 \abs{\bp(d)} \min\set{\sum d, \abs{\bp(d)}})$.
\end{corollary}

Clearly, the above running time becomes polynomial, if $\abs{\bp(d)}$ is polynomial.

Due to \Cref{thm:r-multi-bi-graph}, the minimum $r$ such that a given partition is 
$r$-\maxbi can be computed efficiently implying the following result. 

\begin{corollary}
Let $d$ be a degree sequence of length $n$ such that $\abs{\bp(d)} = \mathcal{O}(n^c)$,
 for some constant $c$. 
Then, $\maxmultbi(d)$ can be computed in polynomial time.
\end{corollary}

Similarly, \Cref{thm:totmult-bi-graph} implies the following.

\begin{corollary}
Let $d$ be a degree sequence of length $n$ such that $\abs{\bp(d)} = \mathcal{O}(n^c)$,
 for some constant $c$. 
Then, $\totmultbi(d) $ can be computed in polynomial time.
\end{corollary}

We remark that a useful special subclass consists of
sequences with a \emph{constant} number of different degrees,
since such a sequence can have at most polynomially many different partitions.

\begin{corollary}
\label{coro:BDR_const_degrees}
Let $q$ be some constant and $d=(d_1^{n_1}, d_2^{n_2},\ldots,d_q^{n_q})$ 
be a degree sequence, where $n = \sum_{i=1}^q n_i$.
Then, $\bp(d)=\mathcal{O}(n^c)$, for some constant $c$.
\end{corollary} 


\section{Small Maximum Degree Sequences}
\label{sec:single-sequence-small-degree}

Towards attacking the realizability problem of general bigraphic sequences,
we look at the question of bounding the total deviation of a nonincreasing
sequence $\seqd=(d_1,\ldots,d_n)$ as a function of its maximum degree, denoted
$\Delta = d_1$.

Burstein and Rubin~\cite{burstein2017sufficient} consider the realization
problem for directed graphs with loops, which is equivalent to BDR$^P$.
(Directed edges go from the first partition to the second.)
They give the following sufficient condition for a pair of sequences to be
the in- and out-degrees of a directed graph with loops.

\begin{theorem}[Burstein and Rubin~\cite{burstein2017sufficient}]
	\label{thm:bur-rub}
Consider a degree sequence $d$ with a partition $(a,b) \in \bp(\seqd)$ 
assuming that $a$ and $b$ have the same length $p$. 
Let $\sum a = \sum b = pc$ where $c$ is the average degree.
If $a_1 b_1 \leq pc +1$, then $d$ is realizable by a directed graph with loops.
\end{theorem}

In what follows we make use of the following straightforward technical claim
which slightly strengthens a similar claim from~\cite{bar2022realizing}.

\begin{observation}
\label{obs:avg}
Consider a nonincreasing integer sequence $d=(d_1,\ldots,d_k)$ 
of total sum $\sum d= D$.
Then, $\sum_{i=1}^\ell d_i \geq \ceil{\ell D/k}$, for every $1 \leq \ell \leq k$.
\end{observation}
\begin{proof}
Since $d$ is nonincreasing, 
$
\frac{1}{\ell} \sum_{i=1}^\ell d_i 
\geq \frac{1}{k - \ell} \sum_{i=\ell+1}^k d_i
~.
$
Consequently, 
\[
D 
=     \sum_{i=1}^k d_i 
=     \sum_{i=1}^\ell d_i + \sum_{i=\ell+1}^k d_i
\leq \sum_{i=1}^\ell d_i + \frac{k-\ell}{\ell} \sum_{i=1}^\ell d_i
=     \frac{k}{\ell} \sum_{i=1}^\ell d_i
~,
\]
implying the claim. 
\end{proof}


\subsection{Bounding the Maximum Multiplicity}

Theorem~\ref{thm:bur-rub} is extended to bipartite multigraphs with
bounded maximum multiplicity, i.e., to $r$-\maxbi sequences. 
The following is a slightly stronger version of Lemma 14 from~\cite{bar2022realizing}.

\begin{lemma}
\label{lem:Delta<sqrtVol-r-bigraphic}
Let $r$ be a positive integer.
Consider a degree sequence $d$ of length $n$ 
with a partition $(a,b) \in \bp(d)$.
If $a_1 \cdot b_1 \leq r \cdot \sum d/2 + r$, then $(a,b)$ is $r$-\maxbi.
\end{lemma}
\begin{proof}
Let $r$, $d$ and $(a,b)$ as in the lemma where $a = (a_1,a_2, \ldots, a_p)$ and $b = (b_1, b_2, \ldots, b_q)$. 
Moreover, let $X = \sum a = \sum b = \sum d / 2$.
To prove the claim, we assume that $a_1 \cdot b_1 \leq r \cdot X + r$,
and show that Equation~\eqref{eq:r-multi-bi-graph}
holds for a fixed index $\ell \in [p]$. 
The lemma then follows due to Theorem~\ref{thm:r-multi-bi-graph}.

First, we consider the case where $b_1 \leq \ell r $. 
Then, 
$\sum_{i= 1}^q \min\{\ell r ,b_i\} = X \geq \sum_{i=1}^{\ell} a_i$, and Equation~\eqref{eq:r-multi-bi-graph} holds.

In the following, we assume that $\ell r < b_1$.
Define the \emph{conjugate} sequence of $b$ as $\conj{b}_j = \abs{\set{b_i \mid b_i \geq j }}$, for every $j$.
Note that the conjugate sequence $\conj{b}$ of $b$ is nonincreasing, and that 
$\sum_{j=1}^{\ell r} \conj{b}_j = \sum_{i= 1}^q \min\{\ell r ,b_i\}$.
%
By Observation~\ref{obs:avg}, 
\[ 
\sum_{i= 1}^q \min\{\ell r ,b_i\}
\geq \ceil{ \ell r X/b_1 }
\geq \ceil{ \ell (a_1 b_1 - r)/b_1 }
= \ceil{\ell a_1 - \ell r/b_1}
=     \ell a_1
~.
\]
As $a$ is nonincreasing, we have that
\(
\sum_{i=1}^{\ell} a_i \leq \ell a_1 \leq \sum_{i= 1}^q \min\{\ell r ,b_i\}
\).
The lemma follows. 
\end{proof}

\begin{lemma}
There exists a degree sequence $d$ with a partition $(a,b) \in \bp(d)$, 
such that $a_1 \cdot b_1 = r \cdot \sum d/2 + r$, which is $r$-\maxbi, 
but not $(r-1)$-\maxbi.
\end{lemma}
\begin{proof}
Consider the sequence $d = (q^{2k-1},(q-1)^2)$ for positive integers $q,k$
such that $q = r \cdot k$.
This sequence has a unique partition $(a,b) \in \bp(d)$, where $a = b = (q^{k-1},(q-1))$.
One can verify that $a_1 b_1 = q^2$, while
\[ \textstyle
r \cdot \sum d/2 + r 
= r (qk-1) + r
= rqk
= q^2
~.
\]
The partition $(a,b)$ is $r$-\maxbi, but no better. 
\end{proof}

Lemma~\ref{lem:Delta<sqrtVol-r-bigraphic} is stated for a given partition (BDR$^P$).
For BDR, we immediately have the following which is a slight improvement over Corollary 16 form~\cite{bar2022realizing}.

\begin{corollary}
Let $r$ be a positive integer and $d$ be a partitionable degree sequence.
If $d_1^2 \leq r \cdot \sum d / 2 + r$, then any partition $(a,b) \in \bp(d)$ is $r$-\maxbi.
\end{corollary}


\subsection{Bounding the Total Multiplicity}

In this section, we establish results for total multiplicity analogous to those obtained 
in the previous section for the maximum multiplicity. 

\begin{lemma}
	\label{lem:Delta<sqrtVol-s-tot-bigraphic}
	Let $t$ be a positive integer.
	Consider a degree sequence $d$ of length $n$ 
	with a partition $(a,b) \in \bp(d)$.
	If $a_1 \cdot b_1 \leq  \sum d/2 + t + 1$, then $(a,b)$ is $t$-\totbi.
\end{lemma}
\begin{proof}
Let $t$, $d$ and $(a,b)$ as in the lemma where $a = (a_1,a_2, \ldots, a_p)$
and $b = (b_1, b_2, \ldots, b_q)$,
and let $X = \sum a = \sum b = \sum d / 2$.
To prove the claim, we assume that $a_1 \cdot b_1 \leq X + t + 1$, 
and show that Equation~\eqref{eq:totmult-bi-graph}
holds for every index $\ell \in [p]$. 
The lemma then follows due to Theorem~\ref{thm:totmult-bi-graph}.
	
First, consider the case where $\ell \geq b_1$. In this case,
\[ 
\sum_{i= 1}^q \min\{\ell ,b_i\} = \sum b = X \geq \sum_{i=1}^{\ell} a_i
~,
\]
and Equation~\eqref{eq:totmult-bi-graph} holds.
	
Next, assume that $\ell < b_1$.
Note that the conjugate sequence $\conj{b}$ of $b$ is nonincreasing, and that 
$\sum_{j=1}^{\ell} \conj{b}_j = \sum_{i= 1}^q \min\{\ell,b_i\}$.
By Observation~\ref{obs:avg}, 
\[ 
\sum_{i= 1}^q \min\{\ell ,b_i\} + t
\geq     \ceil{ \frac{\ell X}{b_1} } + t 
\geq \ceil{ \frac{\ell (a_1 b_1 - t - 1)}{b_1}} + t 
=     \ceil{\ell a_1 - \frac{\ell(t+1)}{b_1}} + t 
\geq \ell a_1
~.
\]
As $a$ is nonincreasing, we have that
\(
\sum_{i=1}^{\ell} a_i \leq \ell a_1 \leq \sum_{i= 1}^q \min\{\ell ,b_i\} + t
\).
The lemma follows.
\end{proof}

The following lemma shows that the above bound it tight.

\begin{lemma}
\label{lem:Delta<sqrtVol-r-tot-bigraphic}
There exists a degree sequence $d$ with a partition $(a,b) \in \bp(d)$, 
such that $a_1 \cdot b_1 = \sum d/2 + t + 2$,
and $(a,b)$ is not $t$-\totbi.
\end{lemma}
\begin{proof}
Consider the sequence $d = (k^{2(k-1)},1^2)$, for a positive integer $k>1$. 
This sequence has only one partition $(a,b) \in \bp(d)$, where $a = b = (k^{k-1},1)$.
Observe that $a_1 b_1 = k^2$, while $\sum d/2 = k(k-1) + 1$.

Assume that $t = k-2$. Hence, $a_1 b_1 = \sum d/2 + t + 1$.
For every $\ell < k$, we have that
\[ 
\sum_{i= 1}^k \min\{\ell ,b_i\} + t
=     k + (\ell-1)(k-1) + k - 2 
=     \ell k - \ell - 1 + k
\geq \ell k
=     \sum_{i=1}^\ell a_i
~.
\]
For $\ell = k$, we have
\[
\sum_{i= 1}^k \min\{\ell ,b_i\} + t
\geq \sum d/2
=     \sum_{i=1}^k a_i
~.
\]

Now assume that $t = k-3$. Hence, $a_1 b_1 = \sum d/2 + t + 2$.
For every $\ell < k$, we have that
\[ 
\sum_{i= 1}^k \min\{\ell ,b_i\} + t
=     k + (\ell-1)(k-1) + k - 3
=     \ell k - \ell - 2 + k
~.
\]
If $\ell = k-1$, we get that
\[ 
\sum_{i=1}^k \min\{\ell ,b_i\} + t
=     (k-1)k - (k-1) - 2 + k
=     (k-1)k - 1
<     \sum_{i=1}^\ell a_i
~,
\]
which means that $(a,b)$ is not $t$-\totbi.
\end{proof}

Similar to above, 
Lemma~\ref{lem:Delta<sqrtVol-r-tot-bigraphic} (stated for BDR$^P$) implies the following for BDR.

\begin{corollary}
Let $t$ be a positive integer and $d$ be a partitionable degree sequence. 
If $d_1^2 \leq \sum d / 2 + t + 1$, then any partition $(a,b) \in \bp(d)$ is $t$-\totbi.
\end{corollary}



\bibliography{realizations}
\bibliographystyle{abbrvnat}

\end{document}